\documentclass{amsart}
\usepackage{amsmath, amsthm, amssymb, amscd, mathrsfs, eucal, epsfig}
\usepackage{pinlabel}

\begin{document}

\newtheorem{theorem}{Theorem}[section]
\newtheorem{lemma}[theorem]{Lemma}
\newtheorem{corollary}[theorem]{Corollary}
\newtheorem{conjecture}[theorem]{Conjecture}
\newtheorem{question}[theorem]{Question}
\newtheorem{problem}[theorem]{Problem}
\newtheorem*{claim}{Claim}
\newtheorem*{criterion}{Criterion}
\newtheorem*{first_thm}{Theorem A}

\theoremstyle{definition}
\newtheorem{definition}[theorem]{Definition}
\newtheorem{construction}[theorem]{Construction}
\newtheorem{notation}[theorem]{Notation}
\newtheorem{convention}[theorem]{Convention}
\newtheorem*{warning}{Warning}

\theoremstyle{remark}
\newtheorem{remark}[theorem]{Remark}
\newtheorem{example}[theorem]{Example}

\def\homeo{\textnormal{Homeo}}
\def\diffeo{\textnormal{Diff}}
\def\bhomeo{\textnormal{BHomeo}}

\def\area{\text{area}}
\def\id{\text{id}}
\def\H{\mathbb H}
\def\Z{\mathbb Z}
\def\C{\mathbb C}
\def\CC{\mathcal C}
\def\R{\mathbb R}
\def\CP{\mathbb {CP}}
\def\RP{\mathbb {RP}}
\def\F{\mathcal F}
\def\G{\mathcal G}
\def\K{\mathcal K}
\def\wind{\textnormal{wind}}
\def\out{\textnormal{Out}}
\def\aut{\textnormal{Aut}}
\def\MCG{\textnormal{MCG}}
\def\fix{\textnormal{fix}}
\def\til{\tilde}
\def\length{\textnormal{length}}
\def\axis{\textnormal{axis}}

\title{The Euler class of planar groups}
\author{Danny Calegari}
\address{Department of Mathematics \\ Caltech \\
Pasadena CA, 91125}
\email{dannyc@its.caltech.edu}

\date{10/10/2008, Version 0.02}

\begin{abstract}
This is an exposition of the homological classification of actions of
surface groups on the plane, in every degree of smoothness.
\end{abstract}

\maketitle

\section{Introduction}

This note gives an exposition of Theorems C and D from the paper
\cite{Calegari}. In particular, it gives the homological classification of
actions of orientation-preserving closed orientable surface groups on the
plane, in every degree of smoothness. I would like to the thank the referee
for reading the paper carefully, and making useful corrections.

\section{Euler class}

\subsection{Homotopy type}

Let $\homeo_+(\R^2)$ denote the group of orientation preserving homeomorphisms
of the plane. This is a topological group with the compact-open topology.

The group $\homeo_+(\R^2)$ with the compact-open topology is homotopy equivalent to $S^1$.
Consequently there are homotopy equivalences of classifying spaces
$$\bhomeo_+(\R^2) \simeq \text{B}S^1 = \CP^\infty$$
and therefore
$$H^*(\bhomeo_+(\R^2);\Z) = H^*(\CP^\infty;\Z) = \Z[e]$$
where $e$ in dimension $2$ generates a free polynomial algebra. The element
$e$ is usually known as the {\em Euler class}.

\begin{remark}
One suggestive way to see this homotopy equivalence is to think of $\R^2$ as
$\CP^1 - \infty$. Given an element $f$ of $\homeo_+(\R^2)$ there is
a unique $\sigma_f \in \C \rtimes \C^*$, the affine group of $\C$, such that
$f_\sigma:=\sigma_f \circ f$ fixes $0,1,\infty$ and therefore induces a map from the
thrice-punctured sphere to itself. If $f$ is sufficiently regular, for instance,
if it is quasiconformal, then there is a Beltrami differential $\mu$ such that
$$\frac {\partial f_\sigma} {\partial \overline{z}} = \mu$$
If $f_\sigma^t$ with $0 \le t \le 1$
denotes the unique quasiconformal self-homeomorphism of $\CP^1$ fixing $0,1,\infty$
with Beltrami differential $(1-t)\mu$, then $f_\sigma^t$ defines a homotopy
in $\homeo_+(\R^2)$ from $f_\sigma$ to $\id$. This demonstrates that
the coset space of $\C \rtimes \C^*$ in the subgroup of quasiconformal
transformations in $\homeo_+(\R^2)$ is contractible, and 
therefore that this subgroup at least has the homotopy type of $\C \rtimes \C^*$, which is to say,
of a circle. Classical arguments can be adapted to show that
the inclusion of the group of quasiconformal homeomorphisms into the group of all
homeomorphisms is a homotopy equivalence (compare with Kneser \cite{Kneser}, \S~2.2 and \S~2.3).
\end{remark}

Given a group $G$, an action of $G$ on $\R^2$ by orientation preserving
homeomorphisms is the same thing as a homomorphism
$$\rho:G \to \homeo_+(\R^2)$$
We can pull back $e$ by $\rho^*$ to an element in group cohomology
$\rho^*e \in H^2(G;\Z)$.

Let $S_g$ denote the closed, orientable surface of genus $g\ge 1$. In the case
$G = \pi_1(S_g)$ the class $\rho^*e$ can be evaluated on the fundamental class
$[S_g]$ which generates $H_2(G;\Z) = \Z$ (implicitly we need to choose an 
orientation on $S_g$ to define $[S_g]$). We call this value the {\em Euler number}
of the action.

The basic question arises as to what Euler numbers can arise for
different $G$, and for different constraints on the analytic quality of
$\rho(G)$.

\begin{remark}\label{circle_example}
The group $\homeo_+(S^1)$ also has the homotopy type of $S^1$, and therefore
we may ask an analogous question for actions of surface groups on circles. In
this case, the well-known {\em Milnor-Wood inequality} is the following inequality.
\begin{theorem}[Milnor-Wood inequality \cite{Milnor,Wood}]\label{Milnor_Wood}
Let $\rho:\pi_1(S_g) \to \homeo_+(S^1)$ be any action, with $g \ge 1$. 
Then there is an inequality
$$|\rho^*e([S_g])| \le -\chi(S_g)$$
\end{theorem}
Moreover it is known that every value which satisfies this inequality can be
realized by an action with any degree of smoothness (in fact, even for analytic
actions).
\end{remark}

\subsection{Calculating $e$}

In this section we give several methods for computing $\rho^*e$. 

\subsubsection{Central extension}
For a group $G$, elements of $H^2(G;\Z)$ classify central extensions of $G$, i.e.
short exact sequences $\Z \to H \to G$ where $\Z$ is central, up to isomorphism.
Let $\G$ denote the group of germs of elements of $\homeo_+(\R^2)$ at infinity.
Every neighborhood of infinity may be restricted further to some annulus neighborhood
$A$. If $\til{A}$ denotes the universal cover of $A$, then elements of
$\G$ can be lifted to germs of homeomorphisms of $\til{A}$. Let $\til{\G}$ denote
the group of lifts of this form. Then there is a central extension
$$\Z \to \til{\G} \to \G$$
where $\Z$ acts on $\til{A}$ as the deck group. The class of this extension
is $e$.

\begin{remark}\label{germ_remark}
Note from this construction that the Euler class depends only on the germ
of a group action at infinity. Another way to see this fact is as follows.
Let $\K$ denote the subgroup of $\homeo_+(\R^2)$ consisting of homeomorphisms
of compact support. Then $\K$ is obviously normal, and there is an
exact sequence
$$\K \to \homeo_+(\R^2) \to \G$$

Moreover, $\K$ is contractible, as can be seen
by the Alexander trick: let $\zeta_t$ be the dilation $\zeta_t:z \to tz$. Then
for any $h \in \K$ the family $h_t: = \zeta_t h \zeta_t^{-1}$ as $t$ goes from
$1$ to $0$ is a path in $\K$ from $h$ to $\id$. This construction shows that
$\K$ is contractible, and therefore that $e$ comes from a class in the
cohomology of $\G$.
\end{remark}

\subsubsection{Bundle}
Given $\rho:\pi_1(S) \to \homeo_+(\R^2)$ one can form a bundle
$$E_\rho = \til{S} \times \R^2/ (s,t) \sim (\alpha(s),\rho(\alpha)(t))$$
over $S = \til{S}/s \sim \alpha(s)$ with fiber $\R^2$. The total space of
$E_\rho$ is a (noncompact) oriented $4$-manifold. Since the fiber $\R^2$ is
contractible, there is a section $\sigma:S \to E_\rho$. The Euler class
of the bundle is the self-intersection number of this section
$$[\sigma(S)] \cap [\sigma(S)] = \rho^*e([S])$$

If the action of $\rho(\pi_1(S))$ is differentiable, the following related
construction makes sense. Let
$D_\rho:\til{S} \to \R^2$ be an equivariant developing map. Since $\R^2$ is
contractible, such a developing map can be constructed skeleton by skeleton
over a fundamental domain for $S$. The vector bundle $T\R^2$ pulls back
to an $\R^2$ (vector space) bundle $\til{E}'_\rho$ over $\til{S}$. Since by
hypothesis the action of $\pi_1(S)$ is differentiable, it acts on $T\R^2$ by
bundle maps, and therefore also on $\til{E}'_\rho$. The quotient $E'_\rho$ is
an $\R^2$ (vector space) bundle over $S$, and $\rho^*e$ is the obstruction to
finding a non-zero section of this bundle.

These two bundles are closely related. The developing map
$D_\rho$ determines and is determined by the section $\sigma$. If the action
is smooth, $\sigma$ can be chosen to be smooth, and $E'_\rho$ is the normal
bundle in $E_\rho$ of $\sigma(S)$.

\subsubsection{Graphical formula}\label{graphical_formula_subsubsection}
Suppose the action of $\pi_1(S_g)$ is at least $C^1$. Let $P$ be a fundamental
polygon for $S$ and $D_\rho\partial P$ the image of the boundary under a $C^1$
developing map which is an immersion on each edge of $\partial P$. 
Let $\delta$ be obtained from $D_\rho\partial P$ by ``smoothing''
the image at the corners; i.e. $\delta$ is the image of the boundary of a regular
neighborhood of $\partial P$ in $P$. Then there is a formula
$$\rho^*e([S_g]) = \wind(\delta) + 1 - 2g$$
where $\wind(\delta)$ denotes the winding number of $\delta$.

To relate this to the previous definition, observe that the immersion defines
a trivialization of $E'_\rho$ over the $1$-skeleton, with a multi-saddle
singularity at the vertex. The winding number of $\delta$ represents the obstruction
to extending this ``trivialization'' over the $2$-skeleton.

\subsection{$C^0$ case}

The following example seems to have been first observed by Bestvina.

\begin{example}[Bestvina]\label{Bestvina_Example}
Let $\tau$ be a Dehn twist in a thin annulus $A$ centered at the origin with radius approximately
$1$. Let $\beta$ be the dilation $z \to 2z$, and suppose that the closure of $A$ is
disjoint from its translates under powers of $\beta$. Define
$$\alpha = \prod_{n \in \Z} \beta^n \tau \beta^{-n}$$
In other words, $\alpha$ is the product of Dehn twists in a family of concentric annuli
all nested about the origin. By construction $[\alpha,\beta] = \id$, so there is
a representation $\rho:\pi_1(S_1) \to \homeo_+(\R^2)$ taking one generator to
$\alpha$ and the other to $\beta$. Observe that $\alpha$ is only $C^0$ at the origin,
but that $\alpha$ and $\beta$ are $C^\infty$ away from the origin.

Notice that $\alpha$ and $\beta$ both fix $0$ so the Euler class represents
the obstruction to lifting $\rho$ to an action on the universal cover of $\R^2 - 0$.
Let $A = \R^2 - 0$ and let $\til{A}$ denote this cover. If we think of $\R^2$ as
$\C$, then we can take $\til{A} = \C$ and the covering map $\til{A} \to A$ can
be taken to be the exponential map. A natural lift of $\beta$ is 
$\til{\beta}: z \to z + \log 2$. However, if
$\til{\alpha}$ is a lift of $\alpha$ which fixes some $z$ for which $e^z$ is not in
the support of $\alpha$, then $\alpha$ acts on the line $z+n\log 2 + it$ (for fixed integer
$n$ and for $t \in \R$) as $z \to z+n2\pi i$. It follows that $[\til{\alpha},\til{\beta}]$
is the translation $z \to z+ 2\pi i$, which is the generator of the deck group of
the covering $\C \to \C^*$. In other words, the value of the Euler class on $[S_1]$ is
$1$.

Replacing $\alpha$ by $\alpha^n$ produces
an action of $\Z \oplus \Z$ with Euler number $n$. Factoring this action with a
degree one map $\pi_1(S_g) \to \pi_1(S_1)$ gives an action of $\pi_1(S_g)$ on $\R^2$
with any Euler number.
\end{example}

In other words:

\begin{theorem}\label{continuous_any}
For any $g \ge 1$ and any $n \in \Z$ there is a representation
$\rho:\pi_1(S_g) \to \homeo_+(\R^2)$ with $\rho^*e([S_g]) = n$.
\end{theorem}

\subsection{$C^\infty$ case}

Consider the following construction.

\begin{example}
Let $\beta$ and $\tau$ be as in Example~\ref{Bestvina_Example}. Define
$$\alpha = \prod_{n \ge 0}\beta^n \tau \beta^{-n}$$
It follows that $[\alpha,\beta] = \tau$. Let $\gamma$ be a homeomorphism
with $\gamma:z \to z+3$ for $\text{real}(z)>-2$ and $\gamma:z \to z$ for
$\text{real}(z) < -3$, and define
$$\delta = \prod_{n \ge 0} \gamma^n \tau \gamma^{-n}$$
Then $[\delta,\gamma]=\tau=[\alpha,\beta]$ so the four elements
$\alpha,\beta,\gamma,\delta$ together define a representation
$\rho:\pi_1(S_2) \to \diffeo_+(\R^2)$.

The image of $[\delta,\gamma] = [\alpha,\beta]$ in $\G$ is trivial, so
the Euler class breaks up into a contribution from two copies of $\Z \oplus \Z$
in $\G$. The first $\Z \oplus \Z$ action has the same image in $\G$ as
the action described in Example~\ref{Bestvina_Example}. The group
generated by $\delta$ and $\gamma$ fixes all points $z$ with
$\text{real}(z)$ sufficiently negative, and such fixed points can be lifted to
give a section of the image in $\G$ to $\til{\G}$, so the Euler class vanishes
on this second copy of $\Z \oplus \Z$. Hence $\rho^*(e)([S_2])=1$ in this example. If
$\tau$ is replaced by $\tau^n$ then $\rho^*(e)([S_2])=n$.
\end{example}

Summarizing:

\begin{theorem}\label{smooth_2_and_more}
For any $g \ge 2$ and any $n \in \Z$ there is a representation
$\rho:\pi_1(S_g) \to \diffeo_+(\R^2)$ with $\rho^*e([S_g]) = n$.
\end{theorem}

It remains to understand what Euler numbers may be realized by
actions of $\pi_1(S_1) = \Z \oplus \Z$ which are $C^1$ or smoother.
We address this in the next section.

\section{$C^1$ actions of $\Z \oplus \Z$}

The main purpose of this section is to prove:

\begin{theorem}\label{torus_smooth_vanishes}
Let $\rho:\pi_1(S_1) \to \diffeo_+^1(\R^2)$ be a $C^1$ action of $\Z \oplus \Z$
on the plane. Then $\rho^*e([S_1]) = 0$.
\end{theorem}

Together with Theorem~\ref{continuous_any} and
Theorem~\ref{smooth_2_and_more} this gives the full homological
classification of actions of surface groups on the plane in any degree of smoothness.

The proof of this theorem divides up into three cases, depending
on the dynamics of the action, and the proof in each case depends on constructions
and techniques particular to that context. We do not know of a single approach
which treats all of the cases on a unified footing. The
cases are as follows:

\begin{enumerate}
\item{One generator (call it $\alpha$) fixes $p$, and the orbit of $p$ under
powers of the other generator (call it $\beta$) is not proper}
\item{$\alpha$ fixes $p$, and the orbit of $p$ under $\beta$ is proper} 
\item{The action is free (i.e. no nontrivial element has a fixed point)}
\end{enumerate}

These cases are treated in subsequent sections.

\subsection{$\alpha$ fixes $p$, the orbit of $p$ under $\beta$ is not proper}

The simplest case is that $\alpha$ and (some power of) $\beta$ have a common
fixed point. Since the Euler class is multiplicative under coverings, we
can assume $\alpha$ and $\beta$ have a common fixed point. Projectivizing the
induced action on the tangent space at this point gives an action on 
the circle $\RP^1$. By the Milnor-Wood inequality (Theorem~\ref{circle_example})
The Euler class of this action vanishes.

Slightly more complicated is the case that $\alpha$ fixes some point $p$, and
there are integers $n_i \to \infty$ so that $p_i:=\beta^{n_i}$ converges to
some point $q$. Since $\alpha$ and $\beta$ commute, it follows that
$\alpha$ fixes each $p_i$, and therefore also $q$.

To analyze this case we introduce some technology which will be useful in what
follows.

\begin{notation}
Given distinct points $a,b \in \R^2$ and an isomorphism $\phi:T_a\R^2 \to T_b\R^2$,
define $\CC(a,b,\phi)$ to be the space of $C^1$ embeddings $f:I \to \R^2$ such that
$f(0) = a$ and $f(1) = b$, and satisfying $\phi f'(0) = f'(1)$.
\end{notation}

If $a,b, \phi$ are understood, we abbreviate $\CC(a,b,\phi)$ by $\CC$.

\begin{lemma}[Whitney]\label{writhe}
The set of path components $\pi_0(\CC)$ is an affine space for $\Z$, where $+1$
acts on $\pi_0(\CC)$ by a positive Dehn twist in a small annulus centered 
at the positive endpoint.
\end{lemma}

This operation is illustrated in Figure~\ref{add_one}.

\begin{figure}[ht]
\labellist
\small\hair 2pt
\pinlabel $\xrightarrow{\quad+1\quad}$ at 300 62
\endlabellist
\centering
\includegraphics[scale=0.4]{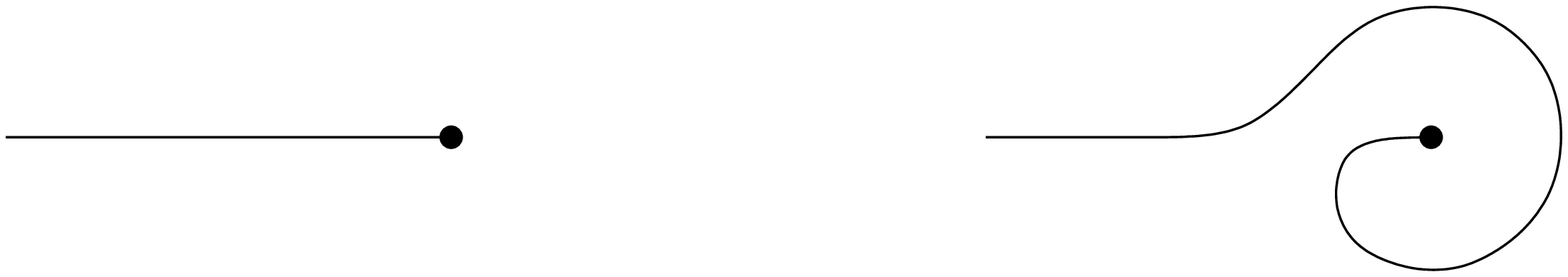}
\caption{The operation of $+1$ on $\pi_0(\CC)$}\label{add_one}
\end{figure}

Given an element $x \in \CC$ the class of $x$ in $\pi_0(\CC)$, denoted
$[x]$, is called the {\em writhe} of $x$.
Since $\pi_0(\CC)$ is an affine space for $\Z$, given elements $x,y \in \CC$ 
the difference of writhes $[x] - [y] \in \Z$ is well-defined.

\begin{lemma}\label{difference_is_Euler}
Let $\alpha,\beta$ be standard generators for $\pi_1(S_1) = \Z \oplus \Z$ acting
on $\R^2$ at least $C^1$ by $\rho$.
Let $p \in \fix(\alpha)$. Let $\tau \in \CC(p,\beta(p),d\beta)$ be
arbitrary. Then $\rho^*e([S_1]) = [\alpha(\tau)] - [\tau]$.
\end{lemma}
\begin{proof}
A $1$-parameter family of smooth curves $\tau_t \in \CC$ from $\tau$ to $\alpha(\tau)$
pulls back to define a trivialization of $E'_\rho$ over a fundamental domain. The
obstruction to this trivialization is $\rho^*e([S_1])$. If $[\alpha(\tau)] - [\tau] = n$
then there is a family $\tau_t$ from $\tau$ to $\tau'$ which differs from $\alpha(\tau)$
only by $n$ twists at one endpoint. This family defines a trivialization of $E'_\rho$
with a singularity of order $n$ at one point.
\end{proof}

\begin{remark}
This also follows from the graphical formula for the Euler class, described in
\S~\ref{graphical_formula_subsubsection}.
\end{remark}

In our context, we have fixed points $p_i,p_j$ of $\alpha$ which are very close to $q$.
Let $\tau$ be a smooth curve from $p_i$ to $p_j$ with $d\beta^{j-i}\tau'(0) = \tau'(1)$.
So $\tau \in \CC(p_i,p_j,d\beta^{j-i})$. Since $\alpha$ is $C^1$,
if $p_i,p_j$ are close enough, the action of $\alpha$ near $q$ is $C^1$ close to a linear
action, so $\tau$ and $\alpha(\tau)$ are $C^1$ close. In particular, $[\alpha(\tau)] - [\tau] = 0$
and therefore the Euler class is zero on the fundamental class of the group generated by
$\alpha$ and $\beta^{j-i}$. But this is $(j-i)$ times $\rho^*e([S_1])$, which must
therefore vanish.

\subsection{$\alpha$ fixes $p$, the orbit of $p$ under $\beta$ is proper}

Let $\tau \in \CC(p,\beta(p),d\beta)$, so that
$\rho^*e([S_1]) = [\alpha(\tau)] - [\tau]$, by Lemma~\ref{difference_is_Euler}.
Choose $\tau$ such that $\tau$ does not intersect $\beta^i(p)$ for any $i$ except
at the endpoints $p$ and $\beta(p)$. Since the orbit $\beta^i(p)$ is proper (and in
any case, countable) this is easy to achieve.

For each integer $i$, we define $a_i$ by the formula
$$a_i = \begin{cases} \alpha(\tau) \cdot \beta^i(\tau) - \tau \cdot \beta^i(\tau) & \text{ if } i \ne 0  \\
0 & \text{ if } i=0 \\ \end{cases}$$
where $\cdot$ denotes intersection number. For $|i|>1$ the endpoints of 
$\alpha(\tau)$ or $\tau$ and of $\beta^i(\tau)$ are disjoint, so this is well-defined.
For $i=1$ we must be careful, since {\it a priori} $\alpha(\tau)$ and $\beta(\tau)$ might
have infinitely many points of intersection near $\beta(p)$, and similarly for $\alpha(\tau)$
and $\beta^{-1}(\tau)$. We replace $\alpha(\tau)$ with $\delta$ which agrees with $\alpha(\tau)$
outside of a small neighborhood of $p \cup \beta(p)$, which agrees with $\tau$ in a smaller
neighborhood of $p \cup \beta(p)$, and which satisfies $[\alpha(\tau)] - [\delta] = 0$.
Notice that $\delta \cdot \beta^i(\tau) = \alpha(\tau) \cdot \beta^i(\tau)$ for
$|i|>1$ if $\alpha(\tau)$ and $\delta$ agree outside a neighborhood of $p \cup \beta(p)$
that does not contain any $\beta^i(p)$ with $i \ne 0,1$.

\begin{lemma}\label{coefficients_vanish}
For sufficiently large $|i|$ we have $a_i = 0$.
\end{lemma}
\begin{proof}
The union $\alpha(\tau) \cup \bar{\tau}$ (i.e. $\tau$ with the opposite orientation)
is a closed, oriented loop. Since the orbit $\beta^i(p)$ is proper, for sufficiently large
$|i|$ the points $\beta^i(p),\beta^{i+1}(p)$ are outside a big disk containing
$\alpha(\tau) \cup \bar{\tau}$ and therefore
$$0 = (\alpha(\tau) \cup \bar{\tau}) \cdot \beta^i(\tau) = \alpha(\tau)\cdot \beta^i(\tau) - \tau \cdot \beta^i(\tau) = a_i$$
\end{proof}

The next Lemma gives a formula for the Euler class in terms of the $a_i$, thus relating
the Euler class to the dynamics of the action.

\begin{lemma}\label{signed_sum_is_euler}
$$\rho^*e([S_1]) = \sum_{i>0} a_i - \sum_{i<0} a_i$$
\end{lemma}
\begin{proof}
Note by Lemma~\ref{coefficients_vanish} that the sum on the right hand side is finite,
and therefore well-defined. By Lemma~\ref{difference_is_Euler}, the left hand side is
equal to the difference in writhes $[\alpha(\tau)] - [\tau]$. Set
$\delta_0 = \delta$ (as in the definition of the $a_i$ above)
and let $\delta_t$ denote a (suitable) $1$-parameter family of curves
in $\CC$ which are constant in a sufficiently small neighborhood of $p$ and $\beta(p)$.
For each $\delta_t$ which does not pass through any $\beta^j(p)$
we can define $a_i(t) := \delta_t \cdot \beta^i(\tau) - \tau \cdot \beta^i(\tau)$ and
therefore the sum $R(t) := \sum_{i > 0} a_i(t) - \sum_{i<0} a_i(t)$.

Since the orbit $\beta^j(p)$ is proper, there are only finitely many values of $t$
when a generic family $\delta_t$ passes through some $\beta^i(p)$. Since each $\delta_t$
is {\em embedded}, it can only pass through $\beta^i(p)$ where $i \ne 0,1$. When it passes
through such a point, $a_i$ changes by $\pm 1$ and $a_{i-1}$ changes by $\mp 1$, so
$R(t)$ stays {\em constant}.

After a suitable deformation, we can assume $\delta_1$ agrees with $\tau$ except in
a small annulus neighborhood of $p$ (or $\beta(p)$) where it differs by the $n$th power of
a Dehn twist, where $n = [\alpha(\tau)] - [\tau]$. It follows that $a_i(1) = 0$ for
all $i \ne 1$, and $a_1(1) = n$. Hence 
$$\rho^*e([S_1]) = [\alpha(\tau)] - [\tau] = n = R(1) = R(0) = \sum_{i>0} a_i - \sum_{i<0} a_i$$
as required.
\end{proof}

The proof in the second case is completed by a covering trick.
Fix some large integer $n$, and define $B = \beta^{2n+1}$. Replace
$\tau$ by $T:= \cup_{i=-n}^{n} \beta^i(\tau)$ and $\alpha(\tau)$ by
$\alpha(T)$. Define coefficients $A_j:= \alpha(T) \cdot B^j(T) - T \cdot B^j(T)$.
The group $\Z \oplus \Z$ generated by $\alpha$ and $B$ is index $2n+1$ in the
group generated by $\alpha$ and $\beta$. It follows from Lemma~\ref{signed_sum_is_euler}
that there is a formula
$$(2n+1)\rho^*e([S_1]) = \sum_{i>0} A_i - \sum_{i<0} A_i$$
However, it is also true that
\begin{align*}
A_j & = \left( \bigcup_{i=-n}^n \alpha\beta^i(\tau) - \beta^i(\tau) \right) \cdot
\left( \bigcup_{i=j(2n+1)-n}^{j(2n+1)+n} \beta^i(\tau) \right)  \\
& = \left( \alpha(\tau) - \tau \right) \cdot \left( \bigcup_{k=-n}^n \bigcup_{i=j(2n+1)-n}^{j(2n+1)+n} \beta^{i-k}(\tau) \right) \\
& = \sum_{i = -2n}^{2n} (2n+1-|i|)(a_{j(2n+1)+i})
\end{align*}
and therefore
$$(2n+1)\rho^*e([S_1]) = \sum_i X_n(i)a_i$$
where
$$X_n(i) = \begin{cases} i & \text{ if } |n|< 2n+1 \\ 2n+1 & \text{ if } n \ge 2n+1 \\ -2n-1 & \text{ if } n \le -2n-1 \\
\end{cases}$$
Part of the graph of $X_3$ is depicted in Figure~\ref{graph}.

\begin{figure}[ht]
\labellist
\small\hair 2pt
%\pinlabel $\xrightarrow{\quad+1\quad}$ at 300 62
\endlabellist
\centering
\includegraphics[scale=0.5]{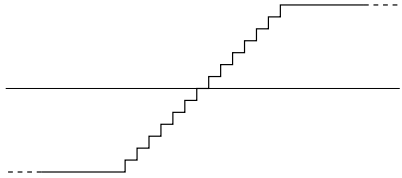}
\caption{The graph of $X_3$}\label{graph}
\end{figure}

Notice that for any $i$ there is an equality $\lim_{n \to \infty} X_n(i) = i$ and therefore
$$\lim_{n \to \infty} (2n+1) \rho^*e([S_1]) = \sum_i ia_i$$
which is a {\em finite sum}, by Lemma~\ref{coefficients_vanish}. In other words, the right hand
side is eventually constant, and therefore the left hand side is identically zero, as we needed to
show.

\subsection{The action is free}

In this case we must cite Theorems of Brouwer and Brown.

\begin{theorem}[Brouwer \cite{Brouwer}]\label{Brouwer_theorem}
If $\alpha \in \homeo^+(\R^2)$ is fixed-point free, the orbit of every point
is proper. Moreover, for any point $p$ there is an arc $\sigma$ from $p$ to $\alpha(p)$
such that the union of the translates $\cup_i \alpha^i(\tau)$ is an embedded copy
of $\R$ in $\R^2$.
\end{theorem}

The caveat to Theorem~\ref{Brouwer_theorem} is that the copy of $\R$ obtained may
not be {\em properly} embedded. 

\begin{definition}
Let $\alpha \in \homeo^+(\R^2)$ be fixed-point free. An embedded arc $\tau$ from $p$ to
$\alpha(p)$ with the property that $\cup_i \alpha^i(\tau)$ is an embedded copy of $\R$
is called a {\em free arc}.
\end{definition}

Brown gave a useful criterion for an arc to be free.

\begin{theorem}[Brown \cite{Brown}]\label{Brown_theorem}
If $\alpha \in \homeo^+(\R^2)$ is fixed-point free, and $\tau$ is an embedded arc
from $p$ to $\alpha(p)$ for which $\tau \cap \alpha(\tau) = \alpha(p)$ then
$\tau$ is a free arc.
\end{theorem}

If $\alpha$ is fixed-point free and $C^1$, then for any point $p$, 
one can choose a free arc in $\CC(p,\alpha(p),d\alpha)$. The next Lemma shows
that the writhe of a free arc is well-defined.

\begin{lemma}\label{writhes_agree}
Let $\tau,\tau' \in \CC(p,\alpha(p),d\alpha)$ be free. Then $[\tau] = [\tau']$.
\end{lemma}
\begin{proof}
For an element $\delta \in \CC(p,\alpha(p),d\alpha)$ define 
$$b_i(\delta) = \delta \cdot \alpha^i(\delta)$$ and
define
$$w(\delta) = \sum_{i \ge 0} b_i(\delta) - \sum_{i \le 0} b_i(\delta)$$
leaving aside for the moment whether this sum is finite. If $\delta$ is free,
then of course $w(\delta) = 0$, since the $b_i$ are termwise zero.
Since $\alpha$ is free, the orbit $\alpha^i(p)$ is proper by Theorem~\ref{Brouwer_theorem}.
If $\delta_t$ is a $1$-parameter family in $\CC$, the value of $w(\delta_t)$ only
changes when $\delta_t$ crosses over some $\alpha^i(p)$, where necessarily $i \ne 0,1$
since elements of $\CC$ are embedded.

At such a crossing, four terms in the sum change: $b_i$ changes by $\pm 1$, $b_{i-1}$
changes by $\mp 1$, $b_{-i}$ changes by $\mp 1$ and $b_{-i+1}$ changes by $\pm 1$.
It follows that $w(\cdot)$ is constant (if it is defined at all) on classes
in $\pi_0(\CC)$.

If we change $\delta$ by performing a small positive Dehn twist at $\alpha(p)$, it
changes the writhe by $+1$. It also changes $b_1$ by $+1$ and $b_{-1}$ by $-1$,
so $w$ changes by $2$. It follows that there are {\em canonical} $\Z$ co-ordinates
on $\pi_0(\CC)$ when $\alpha$ is free, determined by the formula
$$2[\tau] = w(\tau)$$
In particular, $[\tau] - [\tau'] = 0 - 0 = 0$.
\end{proof}

We now return to our original problem. Pick a point $p$, and let
$\tau$ be a smooth arc from $p$ to $\beta(p)$. For each $q \in \tau$
let $\CC_q:= \CC(q,\alpha(q),d\alpha)$. As $q$ varies in $\tau$, this determines
a bundle over $\tau$ which we denote $\CC_\tau$. Since the spaces $\CC_q$
vary continuously as $q$ varies, there is an induced flat connection on
$\pi_0$ of the fibers. Since the base space is an interval, this gives us a
canonical identification $\pi_0(\CC_q) = \pi_0(\CC_r)$ for any two $q,r \in \tau$.

By Theorem~\ref{Brouwer_theorem} and Lemma~\ref{writhes_agree},
every $\pi_0(\CC_q)$ contains a distinguished element,
which contains all the free arcs. By Theorem~\ref{Brown_theorem}, free arcs are
stable under perturbation, so the set of free classes defines a {\em continuous} 
section of the $\pi_0$ bundle. In particular, the canonical identification between
$\pi_0$ of fibers identifies the free class in each fiber.

The element $\beta$ induces a map $\beta:\CC_p \to \CC_{\beta(p)}$ by
$\beta:\sigma \to \beta(\sigma)$ for $\sigma \in \CC_p$. If $\sigma$ is free,
then so is $\beta(\sigma)$, so their respective classes in $\pi_0(\CC_p)$
and $\pi_0(\CC_{\beta(p)})$ are equal under the canonical identification. 
It follows that we can choose a continuous section 
$S:\tau \to \CC_\tau$ with $S(p) = \sigma$ and $S(\beta(p)) = \beta(\sigma)$. The
family of tangents to this family of smooth curves pulls back by $\rho$ to give
a nonvanishing section of $E'_\rho$, and therefore $\rho^*e([S_1])=0$.

This completes the proof of Theorem~\ref{torus_smooth_vanishes}.


\begin{thebibliography}{99}
\bibitem{Brouwer}
	L. E. J. Brouwer,
	\emph{Remark on the plane translation theorem},
	Nederl. Akad. Wetensch. Proc. {\bf 21} (1919) 935--936
\bibitem{Brown}
	M. Brown,
	\emph{Homeomorphisms of two-dimensional manifolds},
	Houston J. Math. {\bf 11} (1985), 455--469
\bibitem{Calegari}
	D. Calegari,
	\emph{Circular groups, planar groups, and the Euler class},
	Geom. Topol. Mon. {\bf 7} (2004), 431--491
\bibitem{Kneser}
	H. Kneser,
	\emph{Deformationss\"atze der einfach zusammenh\"angenden Fl\"achen},
	Math. Zeit. {\bf 25} (1926), 362--372
\bibitem{Milnor}
	J. Milnor,
	\emph{On the existence of a connection with curvature zero},
	Comm. Math. Helv. {\bf 32} (1958), 215--223
\bibitem{Wood}
	J. Wood,
	\emph{Bundles with totally disconnected structure group},
	Comm. Math. Helv. {\bf 46} (1971), 257--273

\end{thebibliography}
\end{document}